\documentclass[preprint,12pt]{elsarticle}

\usepackage{amsmath,amssymb,amsfonts,amsthm,mathtools,bm}
\usepackage{graphicx}
\usepackage{float}
\usepackage{enumitem}
\usepackage[colorlinks=true,linkcolor=blue,citecolor=blue,urlcolor=blue]{hyperref}
\usepackage{mathrsfs}

\journal{Journal of Mathematical Analysis and Applications}
\biboptions{numbers,sort&compress}
\allowdisplaybreaks[3]
\numberwithin{equation}{section}

\newcommand{\R}{\mathbb R}
\newcommand{\C}{\mathbb C}
\newcommand{\ii}{\mathrm i}
\newcommand{\dd}{\,\mathrm d}
\newcommand{\eps}{\varepsilon}
\newcommand{\na}{\nabla}
\newcommand{\pa}{\partial}
\newcommand{\ez}{e_3}

\newcommand{\Ima}{\operatorname{Im}}
\newcommand{\Rea}{\operatorname{Re}}
\newcommand{\Err}{\operatorname{Err}}

\newcommand{\one}{\mathbf 1}
\newcommand{\mopt}{\mathfrak m^{\rm opt}}
\newcommand{\Aopt}{A^{\rm opt}}

\newtheorem{theorem}{Theorem}[section]
\newtheorem{proposition}[theorem]{Proposition}
\newtheorem{lemma}[theorem]{Lemma}
\newtheorem{definition}[theorem]{Definition}
\theoremstyle{remark}
\newtheorem{remark}[theorem]{Remark}

\begin{document}

\begin{frontmatter}

\title{Conditional scattering criteria for cylindrical threshold dynamics of the three-dimensional focusing energy-critical Schr\"odinger equation}

\author[addr1]{Pang-Hung Chung}
\ead{penghongzhong@yahoo.com}
\author[addr2]{Dan Han\corref{cor1}}
\ead{dan.han@louisville.edu}
\cortext[cor1]{Corresponding author.}
\address[addr1]{Department of Applied Mathematics, Guangdong University of Education, Guangzhou 510640, P. R. China}
\address[addr2]{Department of Mathematics, University of Louisville, Louisville, KY 40245, USA}

\begin{abstract}
Below-threshold scattering is studied for the three-dimensional focusing energy-critical Schr\"odinger equation in the cylindrically symmetric class.  For soliton-like compact critical elements, two linked rigidity entrances are isolated: a best fixed-axis window condition for the axial center and a low-frequency tail smallness condition.  A shifted localized virial estimate and a finite-mass virial argument after axial zero-momentum normalization exclude these entrances, yielding a conditional cylindrical threshold scattering criterion via compactness reduction.
\end{abstract}


\end{frontmatter}

\section{Introduction}

We consider the three-dimensional focusing energy-critical Schr\"odinger equation
\begin{equation}\label{eq:nls}
  \ii \pa_t u+\Delta u+|u|^4u=0,
  \qquad (t,x)\in I\times \R^3.
\end{equation}
Its conserved energy is
\[
  E(u)=\frac12\int_{\R^3}|\na u|^2\dd x
       -\frac16\int_{\R^3}|u|^6\dd x.
\]
The criticality comes from the scaling
\[
  u(t,x)\mapsto u_\lambda(t,x)=\lambda^{1/2}u(\lambda^2t,\lambda x),
\]
which preserves both the $\dot H^1(\R^3)$ norm and the energy.  Let the ground state $W$ be the positive solution of
\[
  -\Delta W=W^5,
  \qquad E(W)=\frac13\|\na W\|_2^2,
  \qquad \|W\|_6^6=\|\na W\|_2^2.
\]
Throughout the paper we work under the threshold condition
\begin{equation}\label{eq:threshold}
  E(u_0)<E(W),\qquad \|\na u_0\|_2<\|\na W\|_2.
\end{equation}

The local theory for energy-critical NLS rests on Strichartz estimates and critical-space methods, including restriction estimates, local well-posedness in the critical Sobolev space, and endpoint Strichartz estimates \cite{Strichartz1977,CazenaveWeissler1990,KeelTao1998}.  For a systematic account of nonlinear dispersive equations we refer to \cite{Cazenave2003,Tao2006,LinaresPonce2015}.  The global defocusing energy-critical theory was initiated by Bourgain's radial argument and then developed through interaction Morawetz estimates, induction on energy, and long-time Strichartz analysis, leading to a complete critical scattering theory in three and higher dimensions \cite{Bourgain1999,Visan2007,CollianderKeelStaffilaniTakaokaTao2008}.  In the focusing case the presence of the ground state and unstable bound states makes the threshold geometry subtler.  The compactness-rigidity method of Kenig and Merle gives the radial threshold scattering and blow-up framework; related threshold dynamics and higher-dimensional focusing problems may be found in \cite{KenigMerle2006,DuyckaertsMerle2009,KillipVisan2010,Dodson2019}.  Profile decomposition and concentration compactness are central to this theory; the relevant functional-analytic origins are \cite{Gerard1998,BahouriGerard1999,Keraani2001}, and minimal-counterexample and critical-orbit analyses appear in many NLS problems, for instance \cite{TaoVisanZhang2008}.  Threshold scattering results for the three-dimensional focusing energy-subcritical model also provide important analogies for the rigidity mechanisms used here \cite{HolmerRoudenko2008,DuyckaertsHolmerRoudenko2008}.  We shall also use sharp Sobolev inequalities, Hardy-type estimates, and Fourier analysis; see \cite{Aubin1976,Talenti1976} and \cite{Stein1970,LiebLoss2001,Grafakos2014}.

The value of the unconditional threshold scattering problem is that it seeks to deduce global scattering from the energy and gradient thresholds alone, without imposing radial symmetry, finite mass, negative regularity, or a center-drift control.  If such a result holds, then the ground state $W$ is the only obstruction to the threshold dynamics: below the ground state, on the gradient-small side, a solution cannot form a long-time localized structure and cannot transport its energy along some direction while remaining non-scattering.  Thus unconditional scattering is not merely a technical completion; it is a full dynamical characterization of the focusing critical threshold geometry.

In the radial setting the spatial center of a compact critical element is fixed by symmetry, and localized virial or Morawetz mechanisms can act directly on the principal energy core.  The fully non-radial setting is much harder: the Kenig-Merle reduction usually gives compactness only modulo translations and scalings, and the spatial center may drift in time.  If the virial weight is fixed at the origin, the main energy core may stay away from the truncation region for most times.  The cylindrical class lies between these two regimes: rotational symmetry removes transverse translations, but axial translation remains.  Hence the core obstruction in the cylindrical threshold problem can be formulated more precisely as the control or exclusion of the one-dimensional axial center $z(t)$.

There is a second difficulty.  We work in the critical $\dot H^1$ framework, where neither global mass nor finite variance is available in general.  The rough endpoint estimate for a truncated virial is only $|M_R(t)|\lesssim R^2$.  In order to contradict a positive virial growth of order $T$, the truncation radius must obey $R^2=o(T)$.  On the other hand, $R$ must be large enough to cover both the axial drift scale and the compactness-tail radius.  Thus a key bottleneck for unconditional cylindrical scattering is to show that every possible cylindrical compact critical element either has sufficiently weak axial drift or enjoys an additional low-frequency or mass structure that improves the virial endpoint estimate.

The results of this paper provide several direct entrances around this bottleneck.  First, we do not center the virial weight at the origin.  Instead, on each long time interval we choose a best fixed axis $Z\ez$.  The condition then depends only on the time distribution of $z(t)$ and not on the artificial choice of the coordinate origin.  Second, we allow the bad times to have a small but fixed proportion; a good-bad time decomposition and a rough virial lower bound still preserve average positivity.  This weakens pointwise drift assumptions into quantile-window conditions.  Third, we prove that low-frequency tail smallness yields finite mass and $L^2$ precompactness.  For finite-mass solutions the axial momentum is defined, and an axial Galilean transform can normalize it to zero.  In this zero-momentum gauge, a localized center-of-mass identity yields sublinear center drift.  Consequently, the compact-element exclusion problem is reduced to a few natural and testable axial drift or low-frequency structure assumptions.

We focus on cylindrically symmetric data,
\[
  u(t,x_1,x_2,x_3)=U(t,r,z),
  \qquad r=(x_1^2+x_2^2)^{1/2},\quad z=x_3.
\]
Cylindrical symmetry preserves rotations around the $x_3$-axis but does not remove translations along this axis.  Hence, compared with the radial case, a minimal compact critical element may carry an axial center $z(t)$.  A fixed-origin localized virial loses contact with the main energy core when $z(t)$ is far from the origin.  Our starting point is that the localized virial weight need not be fixed at the origin: on each long interval it may be centered at a fixed axis $Z\ez$.  If the axial center lies near this axis for most times, the bad times can be absorbed by a rough lower bound, leaving a positive averaged virial derivative.

We now formulate compact critical elements.  The compactness reduction will be recalled in Section~\ref{sec:local}.  We only treat the soliton-like branch, namely the branch for which the frequency scale has been normalized to $N(t)\equiv1$.

\begin{definition}[Compact critical element]\label{def:compact-element}
A global cylindrically symmetric solution $u:I\times\R^3\to\C$ satisfying \eqref{eq:threshold} is called a soliton-like cylindrical compact critical element if there exists an axial center function $z:I\to\R$ such that the translated orbit
\[
  K_z=\{u(t,\cdot+z(t)\ez):t\in I\}
\]
is precompact in $\dot H^1(\R^3)$.  Equivalently, for every $\eta>0$ there exists $R_\eta<\infty$ such that
\begin{equation}\label{eq:tightness}
  \sup_{t\in I}\int_{|x-z(t)\ez|\ge R_\eta}
  \bigl(|\na u(t,x)|^2+|u(t,x)|^6\bigr)\dd x\le \eta.
\end{equation}
Unless otherwise stated, compact critical elements in this paper are assumed to be nonzero.
\end{definition}

Given an axial center $z(t)$, define the best fixed-axis bad-time proportion by
\[
  \mopt_T(A)=\inf_{Z\in\R}\frac1T
  \left|\{t\in[0,T]: |z(t)-Z|>A\}\right|,
\]
and define the best $(1-\delta)$-quantile window radius by
\begin{equation}\label{eq:Aopt-def-intro}
  \Aopt_T(\delta)=\inf\{A\ge0:\mopt_T(A)\le\delta\},\qquad 0<\delta<1.
\end{equation}
This quantity records the shortest high-occupancy axial window over a time interval, independently of where the window is placed on the axis.  Let
\[
  K(f)=\int_{\R^3}(|\na f|^2-|f|^6)\dd x.
\]
By the threshold coercivity and nontriviality of the compact element, there is $c_1>0$ such that
\[
  K(u(t))\ge c_1,
  \qquad t\in I.
\]
We shall prove below that there is a rough virial lower-bound constant $C_E<\infty$, depending only on the threshold bounds of the compact element, which controls the negative part of every shifted-center localized virial derivative.  Define
\begin{equation}\label{eq:delta-opt-intro}
  \delta_{\rm opt}=\min\left\{\frac14,\frac{3c_1}{9c_1+C_E}\right\}.
\end{equation}
The first main result is the following.

\begin{theorem}[Window rigidity]\label{thm:window-rigidity}
Let $u$ be a nonzero soliton-like cylindrical compact critical element.  If there exists $\delta_0\in(0,\delta_{\rm opt})$ such that
\begin{equation}\label{eq:Aopt-subdiff-main}
  \liminf_{T\to\infty}\frac{\Aopt_T(\delta_0)}{T^{1/2}}=0,
\end{equation}
then no such $u$ exists.
\end{theorem}

To see the relation between the next theorem and Theorem~\ref{thm:window-rigidity}, temporarily denote by $M_{R,Z}(t)$ the shifted-center localized virial functional to be defined later, where $R$ is the truncation radius and $Z\ez$ is a fixed axis.  The two rigidity entrances are not unrelated statements; they are two realizations of the same endpoint-budget principle.  The localized virial gives a positive linear growth over good times, and it suffices to prove, for two good endpoints,
\[
  |M_{R,Z}(t_1)|+|M_{R,Z}(t_2)|=o(t_2-t_1).
\]
If the endpoint estimate has the form
\[
  |M_{R,Z}(t)|\lesssim R^\beta,
\]
then on an interval of length $T$ one naturally needs $R=o(T^{1/\beta})$.  Theorem~\ref{thm:window-rigidity} corresponds to the pure $\dot H^1$ endpoint budget $\beta=2$, hence to a subdiffusive best axial window.  The low-frequency rigidity theorem improves the budget to $\beta=1$ through finite mass and $L^2$ enhanced compactness, and therefore only requires a sublinear scale.

\begin{table}[H]
\centering
\small
\caption{The common endpoint-budget structure of the two rigidity entrances}
\label{tab:intro-endpoint-budget}
\begin{tabular}{p{0.21\textwidth}p{0.27\textwidth}p{0.18\textwidth}p{0.24\textwidth}}
\hline
Source of information & Endpoint estimate & Closing radius & Corresponding entrance \\ \hline
Pure $\dot H^1$ compactness & $|M_{R,Z}(t)|\lesssim R^2$ & $R=o(T^{1/2})$ & Best fixed-axis window \\
Low-frequency tail smallness & $|M_{R,Z}(t)|\lesssim R$ & $R=o(T)$ & Finite-mass zero-momentum entrance \\ \hline
\end{tabular}
\end{table}

Thus the second result is a low-frequency enhancement of the first.  It does not directly assume that the best window of $z(t)$ is $o(T^{1/2})$.  Instead, it first excludes hidden infinite mass at the zero frequency, obtains a finite-mass endpoint budget, and then uses the axial Galilean zero-momentum gauge and a localized center-of-mass identity to produce a sublinear axial drift.  Table~\ref{tab:intro-endpoint-budget} summarizes the shared budget mechanism.  The axial Galilean zero-momentum gauge means the following.  In the finite-mass case set
\[
  M(u)=\int_{\R^3}|u|^2\dd x,
  \qquad
  P_z(u)=\Ima\int_{\R^3}\bar u\,\pa_z u\dd x,
\]
and for $\xi=\xi_3\ez$ apply
\[
  u^\xi(t,x)=e^{\ii(x\cdot\xi-t|\xi|^2)}u(t,x-2\xi t).
\]
Since $P_z(u^\xi)=P_z(u)+\xi_3M(u)$, choosing $\xi_3=-P_z(u)/M(u)$ yields $P_z(u^\xi)=0$.  This gauge is used only after low-frequency tail smallness has provided finite mass.

Let $P_{\le N}$ denote the Fourier projection
\[
  \widehat{P_{\le N}f}(\xi)=\one_{\{|\xi|\le N\}}\widehat f(\xi).
\]

\begin{theorem}[Low-frequency rigidity]\label{thm:lowfreq-rigidity}
Let $u$ be a nonzero soliton-like cylindrical compact critical element.  If
\begin{equation}\label{eq:lowfreq-small-intro}
  \lim_{N\downarrow0}\sup_{t\in I}\|P_{\le N}u(t)\|_2=0,
\end{equation}
then $u$ has finite mass and the translated orbit $K_z$ is precompact in $L^2(\R^3)$.  Moreover, in the axial Galilean zero-momentum gauge, no such nonzero compact critical element exists.
\end{theorem}

We now combine the two rigidity entrances with the compactness reduction.

\begin{theorem}[Conditional scattering]\label{thm:conditional-scattering}
Let $u_0\in\dot H^1(\R^3)$ be cylindrically symmetric and satisfy \eqref{eq:threshold}.  Let $u:I_{\max}\times\R^3\to\C$ be the corresponding maximal-lifespan solution.  Assume that every nonzero soliton-like cylindrical compact critical element produced by the standard compactness reduction satisfies one of the following conditions:
\begin{enumerate}[label=\textup{(\roman*)}]
\item the best fixed-axis window condition \eqref{eq:Aopt-subdiff-main};
\item the low-frequency tail smallness condition \eqref{eq:lowfreq-small-intro}, and in the finite-mass case the axial Galilean zero-momentum gauge is used.
\end{enumerate}
Then
\begin{equation}\label{eq:global-lifespan-intro}
  I_{\max}=\R,
\end{equation}
\begin{equation}\label{eq:scatter-norm-intro}
  \|u\|_{L^{10}_{t,x}(\R\times\R^3)}<\infty,
\end{equation}
and there exist scattering states $u_\pm\in\dot H^1(\R^3)$ such that
\begin{equation}\label{eq:scattering-states-intro}
  \lim_{t\to\pm\infty}\|u(t)-e^{\ii t\Delta}u_\pm\|_{\dot H^1(\R^3)}=0.
\end{equation}
Equivalently, for each $t\in\R$,
\begin{equation}\label{eq:duhamel-scatter-intro}
  u(t)=e^{\ii t\Delta}u_\pm-\ii\int_t^{\pm\infty}e^{\ii(t-s)\Delta}\bigl(|u(s)|^4u(s)\bigr)\,\dd s
  \quad\text{in }\dot H^1(\R^3).
\end{equation}
\end{theorem}

The proof strategy is as follows.  The first rigidity mechanism uses a shifted-center localized virial.  On good times, the truncated weight captures the main energy core and threshold coercivity gives a positive lower bound for the derivative.  On bad times, we keep only a uniform rough lower bound.  If the bad-time proportion is below the explicit threshold, the virial still grows linearly between two good endpoints.  When the best window radius is $o(T^{1/2})$, the endpoint contribution is $o(T)$, a contradiction.  The second rigidity mechanism starts from low-frequency tail smallness, proves finite mass and $L^2$ precompactness, and then uses a localized center-of-mass identity to derive sublinear center drift for zero axial momentum compact elements.  The finite-mass endpoint estimate then closes the virial contradiction.

The paper is organized as follows.  Section~\ref{sec:local} recalls the standard local theory and compactness reduction.  Section~\ref{sec:virial-prelim} establishes the variational structure and shifted-center localized virial estimates.  Section~\ref{sec:window} proves the best fixed-axis window rigidity.  Section~\ref{sec:lowfreq} treats low-frequency tail smallness, enhanced compactness, and the exclusion of finite-mass zero-momentum compact elements.  Section~\ref{sec:scattering} proves the conditional scattering theorem.

We use the following notation.  The letter $C$ denotes a variable positive constant, and $A\lesssim B$ means $A\le CB$; dependencies are indicated by subscripts, such as $\lesssim_E$ or $\lesssim_u$.  We write $\ez=(0,0,1)$.  Mixed norms are denoted by $\|u\|_{L_t^qL_x^r(I\times\R^3)}$, and the critical scattering norm is $S(I)=L^{10}_{t,x}(I\times\R^3)$.  The Fourier projection $P_{\le N}$ was defined above and $P_{>N}=I-P_{\le N}$.  Section-specific notation, such as virial weights, good-bad time sets, or compact density moduli, will be defined before use.

\section{Local theory and compactness reduction}\label{sec:local}

This section records the standard local theory and explains how failure of scattering in the cylindrically symmetric class produces a soliton-like compact critical element.  We use $S(I)$ as above; $\dot S^1(I)$ denotes the first-order Strichartz control norm and $N^1(I)$ the corresponding dual nonlinear norm.  The following propositions are standard in the energy-critical NLS theory.  We give the precise forms needed here and indicate the references in the proofs.  The new estimates of the paper begin in the next section.

If $\|u\|_{S([0,\infty))}<\infty$, then there exists $u_+\in\dot H^1(\R^3)$ such that
\[
  \lim_{t\to\infty}\|u(t)-e^{\ii t\Delta}u_+\|_{\dot H^1}=0,
\]
and similarly backward in time.  We use the Strichartz control norm
\[
  \|u\|_{\dot S^1(I)}=
  \sup_{(q,r)}\|\na u\|_{L_t^qL_x^r(I\times\R^3)},
\]
where $(q,r)$ ranges over Schr\"odinger admissible pairs in dimension three.  The dual nonlinear norm is denoted by $N^1(I)$.

We begin with the local well-posedness, blow-up criterion, and scattering criterion.

\begin{proposition}[Local theory]\label{prop:LWP}
For every $u_0\in\dot H^1(\R^3)$ there exists a unique maximal-lifespan solution $u:I_{\max}\times\R^3\to\C$.  If $I\Subset I_{\max}$, then $u\in C_t\dot H^1_x(I\times\R^3)$ and $\|u\|_{S(I)}<\infty$.  If $\sup I_{\max}<\infty$, then
\[
  \|u\|_{S([t_0,\sup I_{\max}))}=\infty.
\]
If $u$ is global and $\|u\|_{S(\R)}<\infty$, then $u$ scatters in both time directions.  A maximal-lifespan solution with cylindrically symmetric initial data remains cylindrically symmetric.
\end{proposition}

\begin{proof}
This is the standard local theory for energy-critical NLS.  Local existence and uniqueness follow from Strichartz estimates and the critical nonlinear estimate; see \cite{Strichartz1977,KeelTao1998} for the estimates and \cite{CazenaveWeissler1990,Cazenave2003,Tao2006} for the critical-space framework.  Let
\[
  \Phi_{u_0}(v)(t)=e^{\ii(t-t_0)\Delta}u_0+
  \ii\int_{t_0}^t e^{\ii(t-s)\Delta}(|v|^4v)(s)\dd s.
\]
Strichartz estimates give
\[
  \|\Phi_{u_0}(v)\|_{\dot S^1(I)}
  \lesssim \|u_0\|_{\dot H^1}+
  \|\na(|v|^4v)\|_{N^0(I)}.
\]
The chain rule and H\"older's inequality yield
\[
  \|\na(|v|^4v)\|_{N^0(I)}
  \lesssim \|v\|_{S(I)}^4\|v\|_{\dot S^1(I)}.
\]
If the free evolution is sufficiently small in $S(I)$, the Duhamel map is a contraction.  The maximal-lifespan statement, blow-up criterion, and continuous dependence follow by the standard continuation argument.  If $\|u\|_{S([0,\infty))}<\infty$, then
\[
  e^{-\ii t\Delta}u(t)=u_0+
  \ii\int_0^t e^{-\ii s\Delta}(|u|^4u)(s)\dd s
\]
is Cauchy in $\dot H^1$ and hence converges to a scattering state.  The negative-time direction is identical.  Cylindrical symmetry is preserved by uniqueness.
\end{proof}

The compactness reduction requires comparing approximate solutions built from nonlinear profiles with exact solutions.  We use the following long-time perturbation theorem.

\begin{proposition}[Long-time perturbation]\label{prop:stability}
Let $\widetilde u$ solve on $I$ the approximate equation
\[
  \ii\pa_t\widetilde u+\Delta\widetilde u+|\widetilde u|^4\widetilde u=e,
\]
and assume
\[
  \|\widetilde u\|_{S(I)}\le M,
  \qquad \|\widetilde u\|_{L_t^\infty\dot H_x^1(I\times\R^3)}\le E.
\]
If, for some $t_0\in I$,
\[
  \|e^{\ii(t-t_0)\Delta}(u(t_0)-\widetilde u(t_0))\|_{S(I)}
  +\|\na e\|_{N^0(I)}\le\eps
\]
and $\eps\le\eps_0(M,E)$, then the exact solution $u$ exists on all of $I$ and
\begin{equation}\label{eq:stability-conclusion}
  \|u-\widetilde u\|_{S(I)}+\|u-\widetilde u\|_{\dot S^1(I)}
  \le C(M,E)\eps.
\end{equation}
\end{proposition}

\begin{proof}
This is the standard long-time perturbation theorem for energy-critical NLS; see \cite{KenigMerle2006,CollianderKeelStaffilaniTakaokaTao2008,Tao2006}.  Let $w=u-\widetilde u$.  Then
\[
  \ii\pa_t w+\Delta w+
  \bigl(|u|^4u-|\widetilde u|^4\widetilde u\bigr)=-e.
\]
On a subinterval $J\subset I$, Duhamel's formula and Strichartz estimates give
\begin{align*}
  \|w\|_{S(J)}+\|w\|_{\dot S^1(J)}
  &\lesssim \|e^{\ii(t-t_J)\Delta}w(t_J)\|_{S(J)}
       +\|w(t_J)\|_{\dot H^1} \\
  &\quad +\|\na e\|_{N^0(J)}
       +\|\na(|u|^4u-|\widetilde u|^4\widetilde u)\|_{N^0(J)}.
\end{align*}
Decompose $I$ into finitely many subintervals $J$ such that $\|\widetilde u\|_{S(J)}\le\eta(M,E)$.  On each such interval the chain rule gives
\begin{align*}
  \|\na(|u|^4u-|\widetilde u|^4\widetilde u)\|_{N^0(J)}
  &\lesssim
  (\|u\|_{S(J)}^4+\|\widetilde u\|_{S(J)}^4)\|w\|_{\dot S^1(J)} \\
  &\quad +(\|u\|_{S(J)}^3+\|\widetilde u\|_{S(J)}^3)
     \|w\|_{S(J)}\|\widetilde u\|_{\dot S^1(J)}.
\end{align*}
Taking $\eta$ sufficiently small and iterating a bootstrap from the initial error $\eps$ proves \eqref{eq:stability-conclusion}.  The number of subintervals is controlled by $M$, while the relevant energy-type Strichartz bounds are controlled by $E$ and the local theory.
\end{proof}

We next state the standard minimal-counterexample reduction used in the sequel.

\begin{proposition}[Compactness reduction]\label{prop:compactness-reduction}
Assume that there exists a cylindrically symmetric initial datum $u_0\in\dot H^1(\R^3)$ satisfying \eqref{eq:threshold} whose solution does not scatter.  Then there exists a nonzero soliton-like cylindrical compact critical element $u_c$ in the sense of Definition~\ref{def:compact-element}, with energy $E_c$ satisfying $0<E_c<E(W)$, and with an axial center $z(t)$ for which \eqref{eq:tightness} holds.
\end{proposition}

\begin{proof}
This is the standard Kenig-Merle compactness-rigidity reduction.  The underlying profile decompositions are due to \cite{Gerard1998,BahouriGerard1999,Keraani2001}; the energy-critical NLS minimal-counterexample reduction appears in \cite{KenigMerle2006,CollianderKeelStaffilaniTakaokaTao2008}.  We record the formula-level structure needed below.

Set
\[
  E_c=\inf\{E(v_0): v_0\text{ is cylindrical, satisfies }\eqref{eq:threshold},
  \text{ and its solution does not scatter}\}.
\]
Small-data scattering and the contradiction assumption give
\[
  0<E_c<E(W).
\]
Choose nonscattering data $u_{0,n}$ with
\[
  E(u_{0,n})\downarrow E_c,
  \qquad \|\na u_{0,n}\|_2<\|\na W\|_2.
\]
The $\dot H^1$ linear profile decomposition yields, for each fixed $J$,
\[
  u_{0,n}(x)=\sum_{j=1}^J g_n^j e^{-\ii t_n^j\Delta}\phi^j(x)+w_n^J(x),
\]
where cylindrical symmetry leaves only axial translations, scales, and time translations, and
\[
  (g_n^j f)(x)=(\lambda_n^j)^{-1/2}
  f\left(\frac{x-z_n^j\ez}{\lambda_n^j}\right).
\]
For $j\ne k$, the parameters are orthogonal:
\[
 \frac{\lambda_n^j}{\lambda_n^k}+\frac{\lambda_n^k}{\lambda_n^j}
 +\frac{|t_n^j(\lambda_n^j)^2-t_n^k(\lambda_n^k)^2|}{\lambda_n^j\lambda_n^k}
 +\frac{|z_n^j-z_n^k|^2}{\lambda_n^j\lambda_n^k}\to\infty.
\]
Moreover,
\begin{align}
  \|\na u_{0,n}\|_2^2
  &=\sum_{j=1}^J\|\na\phi^j\|_2^2+\|\na w_n^J\|_2^2+o_n(1),\notag\\
  E(u_{0,n})
  &=\sum_{j=1}^J E(e^{-\ii t_n^j\Delta}\phi^j)+E(w_n^J)+o_n(1),\label{eq:profile-energy-v8}\\
  \lim_{J\to\infty}\limsup_{n\to\infty}
  \|e^{\ii t\Delta}w_n^J\|_{S(\R)}&=0.\notag
\end{align}

Let $U^j$ be the nonlinear profile associated to $\phi^j$ and define
\[
  U_n^j(t,x)=(\lambda_n^j)^{-1/2}
  U^j\left(\frac{t-t_n^j(\lambda_n^j)^2}{(\lambda_n^j)^2},
  \frac{x-z_n^j\ez}{\lambda_n^j}\right).
\]
If all $U^j$ scatter, then orthogonality, perturbation theory, and the smallness of the remainder give the approximate solution
\[
  \widetilde u_n^J(t,x)=\sum_{j=1}^J U_n^j(t,x)+e^{\ii t\Delta}w_n^J(x)
\]
with
\begin{align}
  \sup_{J,n}\|\widetilde u_n^J\|_{S(\R)}&<\infty,\label{eq:approx-S-v8}\\
  \lim_{J\to\infty}\limsup_{n\to\infty}
  \|\na e_n^J\|_{N^0(\R)}&=0,\notag\\
  \lim_{J\to\infty}\limsup_{n\to\infty}
  \|u_{0,n}-\widetilde u_n^J(0)\|_{\dot H^1}&=0,\label{eq:approx-data-v8}
\end{align}
where
\[
  e_n^J=\ii\pa_t\widetilde u_n^J+\Delta\widetilde u_n^J+|\widetilde u_n^J|^4\widetilde u_n^J.
\]
Proposition~\ref{prop:stability} would then imply scattering of the exact solutions with data $u_{0,n}$, a contradiction.  Hence some nonlinear profile $U^{j_0}$ is nonscattering.

By the definition of $E_c$, every threshold-admissible profile with energy below $E_c$ scatters.  Therefore the nonscattering profile satisfies $E(U^{j_0})\ge E_c$.  The energy decoupling \eqref{eq:profile-energy-v8} and $E(u_{0,n})\to E_c$ imply $E(U^{j_0})\le E_c$, hence
\[
  E(U^{j_0})=E_c.
\]
All other nonzero profiles and the remainder cannot carry nonscattering energy.  Taking $u_c=U^{j_0}$ gives a nonscattering solution with energy $E_c$.

It remains to record compactness.  For any time sequence $s_n$, apply profile decomposition to $u_c(s_n)$.  If no modulation parameters $\lambda(s_n)$ and $z(s_n)$ make
\[
  (\lambda(s_n))^{1/2}u_c(s_n,\lambda(s_n)x+z(s_n)\ez)
\]
precompact in $\dot H^1$, a nontrivial profile splitting occurs.  All profiles of energy $<E_c$ scatter, and the same approximate-solution construction together with Proposition~\ref{prop:stability} would force $u_c$ to scatter, again a contradiction.  Consequently,
\[
  \left\{(\lambda(t))^{1/2}u_c(t,\lambda(t)x+z(t)\ez):t\in I\right\}
  \Subset \dot H^1(\R^3).
\]
On the soliton-like branch we normalize $\lambda(t)\equiv1$, obtaining
\[
  \{u_c(t,\cdot+z(t)\ez):t\in I\}\Subset\dot H^1(\R^3),
\]
which is equivalent to the tightness condition \eqref{eq:tightness}.  If a frequency-cascade branch appears in the reduction, it must be excluded by a separate argument in a full unconditional scattering theory; the rigidity arguments below address the soliton-like branch.
\end{proof}

In this reduction, Proposition~\ref{prop:LWP} gives maximal lifespan, the blow-up criterion, and the scattering criterion, while Proposition~\ref{prop:stability} transfers scattering of nonlinear profiles back to the original sequence.  The compact element obtained above is the object of the localized-virial rigidity argument.

\begin{remark}
Proposition~\ref{prop:compactness-reduction} is a standard concentration-compactness input in the Kenig-Merle scheme.  The contribution of the present paper is that, once the compact element satisfying \eqref{eq:tightness} is available, the two natural entrances above can be excluded by a direct localized virial argument.
\end{remark}

\section{Variational structure and shifted-center localized virial}\label{sec:virial-prelim}

This section gives the analytic basis for the paper: threshold coercivity, local conservation laws, the shifted-center localized virial identity, and endpoint estimates.  The functional $K$ is the Pohozaev functional defined in the introduction.  We write $\rho=|u|^2$ and $j=\Ima(\bar u\na u)$.  The truncation radius is $R\ge1$ and the fixed axis is $Z\ez$.

The sharp Sobolev inequality and the ground-state equalities imply that, under \eqref{eq:threshold}, there exists $\delta>0$ such that
\begin{equation}\label{eq:gradient-gap}
  \sup_{t\in I}\|\na u(t)\|_2\le (1-\delta)\|\na W\|_2.
\end{equation}
Thus
\begin{align*}
  \|u(t)\|_6^6
  &\le C_3^3\|\na u(t)\|_2^6
  =\left(\frac{\|\na u(t)\|_2^2}{\|\na W\|_2^2}\right)^2\|\na u(t)\|_2^2
  \le (1-\delta)^4\|\na u(t)\|_2^2.
\end{align*}
Consequently,
\begin{equation}\label{eq:K-coercive}
  K(u(t))=\int_{\R^3}(|\na u(t)|^2-|u(t)|^6)\dd x
  \ge c_0\|\na u(t)\|_2^2,
\end{equation}
where $c_0=1-(1-\delta)^4>0$.  If $u$ is a nonzero compact critical element, then
\[
  \inf_{t\in I}\|\na u(t)\|_2>0.
\]
Indeed, otherwise a sequence $t_n$ with $\|\na u(t_n)\|_2\to0$ would also satisfy $\|u(t_n)\|_6\to0$, hence $E(u)=E(u(t_n))\to0$.  But the threshold gap gives
\[
  E(u(t))\ge\left(\frac12-\frac16(1-\delta)^4\right)\|\na u(t)\|_2^2,
\]
and the coefficient is positive; $E(u)=0$ would force $u\equiv0$.  Thus, by \eqref{eq:K-coercive}, there exists
\begin{equation}\label{eq:c1-positive}
  c_1>0,
  \qquad K(u(t))\ge c_1\quad\text{for all }t\in I.
\end{equation}
This positive lower bound is the source of all virial contradictions below.

The local mass conservation law is
\begin{equation}\label{eq:mass-local}
  \pa_t\rho+2\operatorname{div}j=0.
\end{equation}
If $u\in H^1$, then the mass and the axial momentum
\[
  P_z(u)=\Ima\int_{\R^3}\bar u\pa_z u\dd x
\]
are conserved.  The local momentum conservation law is
\begin{equation}\label{eq:momentum-local}
  \pa_t j_k+\pa_\ell S_{k\ell}=0,
\end{equation}
where
\[
  S_{k\ell}=2\Rea(\pa_k\bar u\,\pa_\ell u)-\frac12\delta_{k\ell}\Delta\rho-
  \frac23\delta_{k\ell}\rho^3.
\]
For a real-valued $a\in C^4$, set
\[
  M_a(t)=2\Ima\int_{\R^3}\na a(x)\cdot\na u(t,x)\,\overline{u(t,x)}\dd x.
\]
Integrating by parts in \eqref{eq:momentum-local} gives
\begin{equation}\label{eq:local-virial-general}
  \frac{\dd}{\dd t}M_a(t)
  =4\int a_{k\ell}\Rea(\pa_k u\,\overline{\pa_\ell u})\dd x
  -\int (\Delta^2a)|u|^2\dd x
  -\frac43\int(\Delta a)|u|^6\dd x.
\end{equation}
For $a(x)=|x|^2/2$ this becomes
\[
  \frac{\dd}{\dd t}2\Ima\int x\cdot\na u\,\bar u\dd x=4K(u(t)).
\]
Since a compact element need not have finite variance, we use only truncated versions.

Fix $a\in C^4([0,\infty))$ such that
\[
  a(s)=\frac12s^2\quad(0\le s\le1),
  \qquad a'(s)=0\quad(s\ge2),
  \qquad 0\le a''(s)\le1,
\]
with all relevant derivatives bounded.  For $R\ge1$ and $Z\in\R$, define
\[
  a_{R,Z}(x)=R^2a\left(\frac{|x-Z\ez|}{R}\right),
\]
and
\[
  M_{R,Z}(t)=2\Ima\int_{\R^3}\na a_{R,Z}(x)\cdot\na u(t,x)\overline{u(t,x)}\dd x.
\]
We write $M_R=M_{R,0}$.  The cutoff satisfies
\begin{equation}\label{eq:cutoff-bounds}
  |\na a_{R,Z}|\le CR\one_{\{|x-Z\ez|\le2R\}},\quad
  |\na^2a_{R,Z}|\le C,
\end{equation}
and
\begin{equation}\label{eq:bi-lap-bound}
  |\Delta^2a_{R,Z}|\le CR^{-2}\one_{\{R\le |x-Z\ez|\le2R\}}.
\end{equation}
On the ball $|x-Z\ez|\le R$, one has $\na^2a_{R,Z}=I_3$, $\Delta a_{R,Z}=3$, and $\Delta^2a_{R,Z}=0$.  Inserting this into \eqref{eq:local-virial-general} gives
\begin{equation}\label{eq:MRZ-derivative}
  M'_{R,Z}(t)=4K(u(t))+\Err_{R,Z}(t).
\end{equation}
More explicitly,
\begin{align*}
  \Err_{R,Z}(t)
  &=4\int (a_{R,Z;k\ell}-\delta_{k\ell})
        \Rea(\pa_k u\,\overline{\pa_\ell u})\dd x \\
  &\quad -\frac43\int(\Delta a_{R,Z}-3)|u|^6\dd x
        -\int(\Delta^2a_{R,Z})|u|^2\dd x.
\end{align*}
Hence
\begin{equation}\label{eq:error-prehardy}
  |\Err_{R,Z}(t)|\le C\int_{|x-Z\ez|\ge R}
  (|\na u(t)|^2+|u(t)|^6)\dd x
  +CR^{-2}\int_{R\le |x-Z\ez|\le2R}|u(t)|^2\dd x.
\end{equation}

We first remove the local $L^2$ term by a translated Hardy inequality.

\begin{lemma}[Translated Hardy]\label{lem:shifted-hardy}
For any $a_0\in\R^3$ and $f\in\dot H^1(\R^3)$,
\begin{equation}\label{eq:shifted-hardy}
  \int_{\R^3}\frac{|f(x)|^2}{|x-a_0|^2}\dd x
  \le4\int_{\R^3}|\na f(x)|^2\dd x.
\end{equation}
Consequently,
\begin{equation}\label{eq:annulus-hardy}
  R^{-2}\int_{R\le |x-a_0|\le2R}|f(x)|^2\dd x
  \le C\int_{|x-a_0|\ge R/2}|\na f(x)|^2\dd x.
\end{equation}
\end{lemma}

\begin{proof}
The estimate \eqref{eq:shifted-hardy} is the usual three-dimensional Hardy inequality applied to $f(\cdot+a_0)$.  For the localized estimate, choose $\chi\in C^\infty([0,\infty))$ with $\chi=0$ on $[0,1/2]$ and $\chi=1$ on $[1,\infty)$ and set
\[
  g(x)=\chi\left(\frac{|x-a_0|}{R}\right)f(x).
\]
Since $g=f$ on $|x-a_0|\ge R$,
\begin{align*}
  R^{-2}\int_{R\le |x-a_0|\le2R}|f|^2\dd x
  &\le C\int_{|x-a_0|\ge R}\frac{|g(x)|^2}{|x-a_0|^2}\dd x
  \le C\int_{\R^3}|\na g(x)|^2\dd x.
\end{align*}
Moreover,
\[
  |\na g|^2\le C\one_{\{|x-a_0|\ge R/2\}}|\na f|^2
  +CR^{-2}\one_{\{R/2\le |x-a_0|\le R\}}|f|^2.
\]
The last term is absorbed by applying the same estimate on adjacent dyadic annuli, equivalently by the standard exterior Hardy inequality
\[
  \int_{|x-a_0|\ge R}\frac{|f|^2}{|x-a_0|^2}\dd x
  \le C\int_{|x-a_0|\ge R/2}|\na f|^2\dd x.
\]
This proves \eqref{eq:annulus-hardy}.
\end{proof}

By Lemma~\ref{lem:shifted-hardy}, \eqref{eq:error-prehardy} becomes
\begin{equation}\label{eq:error-hardied}
  |\Err_{R,Z}(t)|\le C\int_{|x-Z\ez|\ge R/2}
  (|\na u(t)|^2+|u(t)|^6)\dd x.
\end{equation}

\begin{figure}[H]
\centering
\includegraphics[width=0.62\textwidth]{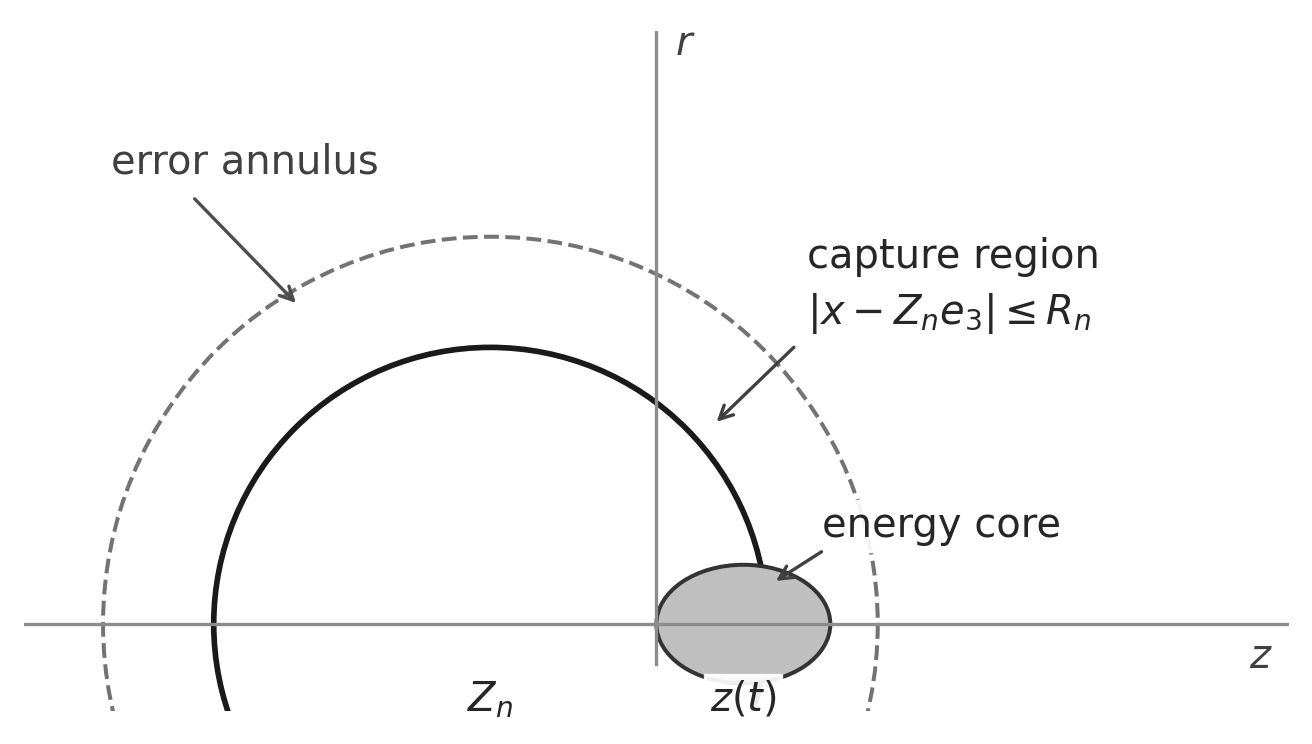}
\caption{Geometry of the shifted-center localized virial.  The truncation is centered at the fixed axis $Z_n e_3$; the main energy core lies inside the capture region, while the error comes from the exterior annulus.  Fixing $Z_n$ on a long time interval avoids derivative terms generated by a time-dependent center.}
\label{fig:shifted-virial}
\end{figure}

Figure~\ref{fig:shifted-virial} shows the covering relation behind the shifted truncation.  When the fixed-axis window covers the main energy core, the previous error estimate yields a positive derivative.

\begin{lemma}[Good-time positivity]\label{lem:good-time-positive}
Let $u$ be a nonzero compact critical element.  For $\eta>0$ sufficiently small, if
\[
  |z(t)-Z|\le A,
  \qquad R\ge 4(R_\eta+A+1),
\]
then
\[
  M'_{R,Z}(t)\ge3c_1.
\]
\end{lemma}

\begin{proof}
If $|x-Z\ez|\ge R/2$, then
\[
  |x-z(t)\ez|\ge |x-Z\ez|-|Z-z(t)|\ge R/2-A\ge R_\eta.
\]
Thus \eqref{eq:tightness} and \eqref{eq:error-hardied} give $|\Err_{R,Z}(t)|\le C\eta$.  Choose $\eta$ such that $C\eta\le c_1$ and use \eqref{eq:MRZ-derivative} together with \eqref{eq:c1-positive}.
\end{proof}

On bad times the truncation need not cover the energy core.  We still need a uniform lower bound.

\begin{lemma}[Rough lower bound]\label{lem:rough-lower}
There exists $C_E<\infty$ such that for all $R\ge1$, $Z\in\R$, and $t\in I$,
\begin{equation}\label{eq:rough-derivative-bound}
  |M'_{R,Z}(t)|\le C_E,
  \qquad M'_{R,Z}(t)\ge -C_E.
\end{equation}
\end{lemma}

\begin{proof}
Using \eqref{eq:local-virial-general}, \eqref{eq:cutoff-bounds}, \eqref{eq:bi-lap-bound}, Hardy's inequality, Sobolev embedding, and the threshold gradient bound,
\begin{align*}
  |M'_{R,Z}(t)|
  &\le C\int_{|x-Z\ez|\le2R}|\na u(t)|^2\dd x
     +C\int_{|x-Z\ez|\le2R}|u(t)|^6\dd x \\
  &\quad +CR^{-2}\int_{R\le |x-Z\ez|\le2R}|u(t)|^2\dd x \\
  &\le C\|\na u(t)\|_2^2+C\|\na u(t)\|_2^6
  \le C_E.
\end{align*}
\end{proof}

The rough endpoint estimate uses only $\dot H^1\hookrightarrow L^6$.  From \eqref{eq:cutoff-bounds}, Cauchy-Schwarz, H\"older, and Sobolev,
\begin{align*}
  |M_{R,Z}(t)|
  &\le C R\|\na u(t)\|_2
  \left(\int_{|x-Z\ez|\le2R}|u(t,x)|^2\dd x\right)^{1/2} \\
  &\le C R\|\na u(t)\|_2 |B(0,2R)|^{1/3}\|u(t)\|_6
  \le C_uR^2.
\end{align*}
Thus
\begin{equation}\label{eq:coarse-endpoint}
  |M_{R,Z}(t)|\le C_uR^2.
\end{equation}
If finite mass is available, then
\begin{equation}\label{eq:mass-endpoint}
  |M_{R,Z}(t)|\le C R\|u(t)\|_2\|\na u(t)\|_2\le C_uR.
\end{equation}

Define the compact density modulus
\[
  \omega_u(R)=R^{-2}\sup_{t\in I}\int_{|x-z(t)\ez|\le R}|u(t,x)|^2\dd x,
  \qquad R\ge1.
\]
This modulus improves the pure $\dot H^1$ endpoint estimate and decays automatically under compactness.

\begin{lemma}[Density modulus decay]\label{lem:omega-decay}
If $u$ satisfies \eqref{eq:tightness}, then
\[
  \omega_u(R)\to0,
  \qquad R\to\infty.
\]
\end{lemma}

\begin{proof}
Fix $\eps>0$.  By \eqref{eq:tightness}, choose $L=L_\eps$ such that
\[
  \sup_t\int_{|x-z(t)\ez|\ge L}|u(t,x)|^6\dd x\le\eps.
\]
For $R\ge2L$, split the ball $\{|x-z(t)\ez|\le R\}$ into the inner ball of radius $L$ and the annulus $L<|x-z(t)\ez|\le R$.  The inner part is bounded by $C_uL^2$ by H\"older and Sobolev.  The annular part satisfies
\[
  \int_{L<|x-z(t)\ez|\le R}|u(t)|^2\dd x
  \le |B(0,R)|^{2/3}
  \left(\int_{|x-z(t)\ez|\ge L}|u(t)|^6\dd x\right)^{1/3}
  \le CR^2\eps^{1/3}.
\]
Hence $\omega_u(R)\le C_uL^2/R^2+C\eps^{1/3}$.  Let $R\to\infty$ and then $\eps\downarrow0$.
\end{proof}

Lemma~\ref{lem:omega-decay} says that $\dot H^1$ compactness automatically rules out exact quadratic growth of the local mass.  At endpoints lying inside the fixed-axis window, this improves the virial endpoint estimate.

\begin{lemma}[Density endpoint]\label{lem:omega-endpoint}
If $|z(t)-Z|\le A$ and $A\le R/4$, then
\[
  |M_{R,Z}(t)|\le C_uR^2\omega_u(3R)^{1/2}.
\]
\end{lemma}

\begin{proof}
The condition implies
\[
  \{|x-Z\ez|\le2R\}\subset\{|x-z(t)\ez|\le3R\}.
\]
Thus
\begin{align*}
  |M_{R,Z}(t)|
  &\le C R\|\na u(t)\|_2
  \left(\int_{|x-Z\ez|\le2R}|u(t,x)|^2\dd x\right)^{1/2} \\
  &\le C_u R
  \left(\int_{|x-z(t)\ez|\le3R}|u(t,x)|^2\dd x\right)^{1/2}
  \le C_uR^2\omega_u(3R)^{1/2}.
\end{align*}
\end{proof}

\begin{remark}[Morrey endpoint]\label{rem:morrey}
If for some $0\le\sigma<1$ and $C_\sigma$,
\begin{equation}\label{eq:morrey-mass}
  \int_{|x-z(t)\ez|\le R}|u(t,x)|^2\dd x\le C_\sigma R^{2\sigma},
  \qquad R\ge1,
\end{equation}
then the same endpoint computation gives $|M_{R,Z}(t)|\lesssim R^{1+\sigma}$.  This interpolates between the finite-mass endpoint $\sigma=0$ and the pure $\dot H^1$ endpoint $\sigma=1$.
\end{remark}

\section{Best fixed-axis window rigidity}\label{sec:window}

This section proves Theorem~\ref{thm:window-rigidity} and records several equivalent or stronger forms.  The sets $G_n$ and $B_n$ will denote good times and bad times, respectively; $A_n$ is the window radius and $Z_n\ez$ the fixed axis chosen on the interval.  The idea is to choose a constant axis on each long interval, use positivity on the good times, and use only the rough lower bound on the bad times.

\begin{figure}[H]
\centering
\includegraphics[width=0.72\textwidth]{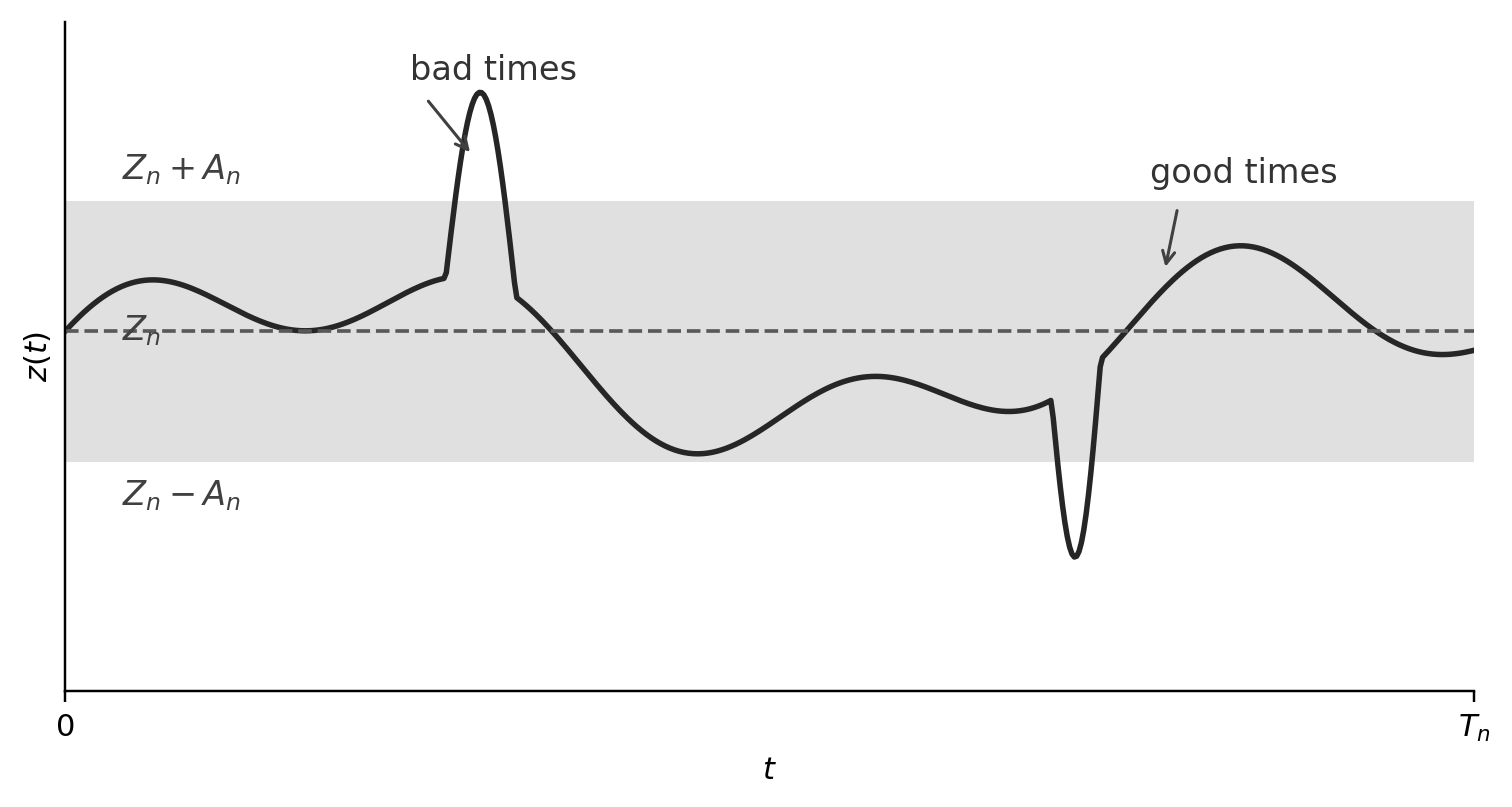}
\caption{Time distribution associated with the best fixed-axis window.  The shaded band represents the fixed window $|z(t)-Z_n|\le A_n$; most times lie inside it, while only a small set of bad times leaves the window.  Theorem~\ref{thm:window-rigidity} converts this time-occupancy property into positive averaged virial growth.}
\label{fig:window-band}
\end{figure}

Figure~\ref{fig:window-band} depicts the good-time window.  We first prove a version with the density modulus; the main theorem follows from the rough endpoint estimate.

\begin{theorem}[Window modulus]\label{thm:window-omega}
Let $u$ be a nonzero soliton-like cylindrical compact critical element.  Let $c_1$, $C_E$, and $\delta_{\rm opt}$ be as in \eqref{eq:c1-positive}, \eqref{eq:rough-derivative-bound}, and \eqref{eq:delta-opt-intro}.  Assume that there exist $\delta_0\in(0,\delta_{\rm opt})$, $T_n\to\infty$, $A_n\ge0$, $Z_n\in\R$, and $\eta_n\downarrow0$ such that
\begin{equation}\label{eq:bad-time-window}
  \left|\{t\in[0,T_n]: |z(t)-Z_n|>A_n\}\right|\le\delta_0T_n,
\end{equation}
and let
\begin{equation}\label{eq:Rn-window}
  R_n=4(R_{\eta_n}+A_n+1)+T_n^{1/4}.
\end{equation}
If
\begin{equation}\label{eq:omega-compatible}
  \frac{R_n^2\omega_u(3R_n)^{1/2}}{T_n}\to0,
\end{equation}
then no such $u$ exists.
\end{theorem}

\begin{proof}
Choose $\alpha\in(\delta_0,1/2)$ such that
\[
  \gamma:=3c_1(1-2\alpha)-(3c_1+C_E)\delta_0>0.
\]
This is possible because as $\alpha\downarrow\delta_0$ the right-hand side tends to $3c_1-(9c_1+C_E)\delta_0>0$.  Define
\[
  G_n=\{t\in[0,T_n]: |z(t)-Z_n|\le A_n\},
  \qquad B_n=[0,T_n]\setminus G_n.
\]
Since $|B_n|\le\delta_0T_n<\alpha T_n$, choose
\[
  t_{1,n}\in[0,\alpha T_n]\cap G_n,
  \qquad t_{2,n}\in[(1-\alpha)T_n,T_n]\cap G_n.
\]
Then $t_{2,n}-t_{1,n}\ge(1-2\alpha)T_n$.  For $t\in G_n$, \eqref{eq:Rn-window} gives
\[
  R_n\ge4(R_{\eta_n}+|z(t)-Z_n|+1),
\]
so Lemma~\ref{lem:good-time-positive} yields
\[
  M'_{R_n,Z_n}(t)\ge3c_1.
\]
For $t\in B_n$, Lemma~\ref{lem:rough-lower} gives $M'_{R_n,Z_n}(t)\ge-C_E$.  Hence
\begin{align*}
  M_{R_n,Z_n}(t_{2,n})-M_{R_n,Z_n}(t_{1,n})
  &\ge 3c_1|[t_{1,n},t_{2,n}]\cap G_n|
       -C_E|[t_{1,n},t_{2,n}]\cap B_n| \\
  &=3c_1(t_{2,n}-t_{1,n})-(3c_1+C_E)|[t_{1,n},t_{2,n}]\cap B_n| \\
  &\ge \gamma T_n.
\end{align*}
At the endpoints $t_{i,n}\in G_n$, so $|z(t_{i,n})-Z_n|\le A_n\le R_n/4$.  Lemma~\ref{lem:omega-endpoint} and \eqref{eq:omega-compatible} imply
\[
  |M_{R_n,Z_n}(t_{2,n})|+|M_{R_n,Z_n}(t_{1,n})|
  \le C_uR_n^2\omega_u(3R_n)^{1/2}=o(T_n),
\]
contradicting the lower bound $\gamma T_n$.
\end{proof}

A simpler version follows by using the rough endpoint estimate only.

\begin{proposition}[Rough endpoint]\label{prop:coarse-window}
In Theorem~\ref{thm:window-omega}, condition \eqref{eq:omega-compatible} may be replaced by
\begin{equation}\label{eq:coarse-compatible}
  \frac{R_n^2}{T_n}\to0.
\end{equation}
Then the same nonexistence conclusion holds.
\end{proposition}

\begin{proof}
The positive-growth part of the proof of Theorem~\ref{thm:window-omega} is unchanged and gives
\[
  M_{R_n,Z_n}(t_{2,n})-M_{R_n,Z_n}(t_{1,n})\ge \gamma T_n.
\]
Using \eqref{eq:coarse-endpoint} instead, the endpoints satisfy
\[
  |M_{R_n,Z_n}(t_{2,n})|+|M_{R_n,Z_n}(t_{1,n})|
  \le 2C_uR_n^2=o(T_n),
\]
which is impossible.
\end{proof}

\begin{proof}[Proof of Theorem~\ref{thm:window-rigidity}]
By \eqref{eq:Aopt-subdiff-main}, choose $S_n\to\infty$ such that
\begin{equation}\label{eq:Aopt-seq-small}
  \Aopt_{S_n}(\delta_0)S_n^{-1/2}\to0.
\end{equation}
For each $j$, take $\eta_j\downarrow0$.  Passing to a subsequence $T_j=S_{n_j}$, we may require
\begin{equation}\label{eq:subsequence-dominates-tightness}
  \frac{R_{\eta_j}^2}{T_j}\to0,
  \qquad
  \frac{\Aopt_{T_j}(\delta_0)^2}{T_j}\to0.
\end{equation}
By the definition of $\Aopt$, choose $A_j\le \Aopt_{T_j}(\delta_0)+1$ and $Z_j\in\R$ such that
\[
  \left|\{t\in[0,T_j]: |z(t)-Z_j|>A_j\}\right|\le(\delta_0+o(1))T_j.
\]
Increasing $\delta_0$ slightly while keeping it below $\delta_{\rm opt}$, the bad-time condition of Proposition~\ref{prop:coarse-window} holds.  With
\[
  R_j=4(R_{\eta_j}+A_j+1)+T_j^{1/4},
\]
\eqref{eq:subsequence-dominates-tightness} gives $R_j^2/T_j\to0$.  Proposition~\ref{prop:coarse-window} excludes $u$.
\end{proof}

We record several useful reformulations of the same window condition.

\begin{proposition}[Occupation time]\label{prop:occupation-window}
Let $u$ be a nonzero soliton-like cylindrical compact critical element.  Suppose that for some $\delta_0\in(0,\delta_{\rm opt})$ there are $T_n\to\infty$, $A_n=o(T_n^{1/2})$, and $Z_n\in\R$ such that
\[
  |\{t\in[0,T_n]: |z(t)-Z_n|>A_n\}|\le\delta_0T_n.
\]
Then no such $u$ exists.
\end{proposition}

\begin{proof}
Choose $\eta_n\downarrow0$ and pass to a subsequence with $R_{\eta_n}^2/T_n\to0$.  Let $R_n=4(R_{\eta_n}+A_n+1)+T_n^{1/4}$.  Then $R_n^2/T_n\to0$, and Proposition~\ref{prop:coarse-window} applies.
\end{proof}

The occupation-time formulation contains pointwise subdiffusive drift and averaged drift conditions.

\begin{proposition}[Drift forms]\label{prop:drift-criteria}
Each of the following conditions excludes a nonzero soliton-like cylindrical compact critical element:
\begin{enumerate}[label=\textup{(\alph*)}]
\item there exist $T_n\to\infty$ and $Z_n\in\R$ such that
\[
  \sup_{0\le t\le T_n}|z(t)-Z_n|=o(T_n^{1/2});
\]
\item there exist $p>0$ and $Z_T\in\R$ such that
\begin{equation}\label{eq:Lp-drift-condition}
  \int_0^T |z(t)-Z_T|^p\dd t=o(T^{1+p/2}).
\end{equation}
\end{enumerate}
In particular, $p=2$ gives a time-averaged subdiffusive drift condition.
\end{proposition}

\begin{proof}
For (a), take $A_n=\sup_{0\le t\le T_n}|z(t)-Z_n|$.  The bad-time set is empty and Proposition~\ref{prop:occupation-window} applies.  For (b), fix $\delta_0\in(0,\delta_{\rm opt})$ and set
\[
  A_T=\left(\frac{1}{\delta_0T}\int_0^T|z(t)-Z_T|^p\dd t\right)^{1/p}.
\]
Markov's inequality gives $|\{t: |z(t)-Z_T|>A_T\}|\le\delta_0T$, and \eqref{eq:Lp-drift-condition} gives $A_T=o(T^{1/2})$.  Proposition~\ref{prop:occupation-window} applies.
\end{proof}

Proposition~\ref{prop:drift-criteria} is a convenient reformulation of Theorem~\ref{thm:window-rigidity} in pointwise and averaged terms.  The next statement uses the Morrey endpoint from Remark~\ref{rem:morrey}.

\begin{proposition}[Morrey window]\label{prop:morrey-window}
Let $u$ be a nonzero soliton-like cylindrical compact critical element satisfying the Morrey local-mass growth \eqref{eq:morrey-mass} for some $0\le\sigma<1$.  Suppose that there exist $\delta_0\in(0,\delta_{\rm opt})$, $T_n\to\infty$, $A_n\ge0$, $Z_n\in\R$, and $\eta_n\downarrow0$ such that \eqref{eq:bad-time-window} holds and
\[
  R_n=4(R_{\eta_n}+A_n+1)+T_n^{1/4},
  \qquad \frac{R_n^{1+\sigma}}{T_n}\to0.
\]
Then no such $u$ exists.
\end{proposition}

\begin{proof}
The positive virial-growth integral is the same as in Theorem~\ref{thm:window-omega}.  It remains to check the endpoints.  For $t_{i,n}\in G_n$,
\[
  \{|x-Z_n\ez|\le2R_n\}\subset\{|x-z(t_{i,n})\ez|\le3R_n\}.
\]
By the definition of $M_{R,Z}$, Cauchy-Schwarz, and \eqref{eq:morrey-mass},
\begin{align*}
  |M_{R_n,Z_n}(t_{i,n})|
  &\le C R_n\|\na u(t_{i,n})\|_2
      \left(\int_{|x-Z_n\ez|\le2R_n}|u(t_{i,n})|^2\dd x\right)^{1/2} \\
  &\le C_u R_n
      \left(\int_{|x-z(t_{i,n})\ez|\le3R_n}|u(t_{i,n})|^2\dd x\right)^{1/2}
  \le C_u R_n^{1+\sigma}.
\end{align*}
The endpoint sum is $o(T_n)$, contradicting the linear virial growth.
\end{proof}

Proposition~\ref{prop:morrey-window} is the same endpoint-budget principle under Morrey mass growth.  Together with Theorem~\ref{thm:window-omega}, it shows that the window mechanism adjusts continuously to the available local-mass information.

\begin{remark}[Fixed axis]
If the virial center were taken to be the time-dependent center $z(t)$, derivative terms involving $z'(t)$ would appear.  In the compactness reduction, $z(t)$ is usually only a modulation center and need not be sufficiently regular.  We avoid this by choosing a constant axis $Z_n$ on each interval $[0,T_n]$ and absorbing the times at which $z(t)$ leaves the corresponding window.  Thus the argument requires only time-distribution information for $z(t)$ and not any derivative of $z(t)$.
\end{remark}

\section{Low-frequency tail smallness and enhanced compactness}\label{sec:lowfreq}

This section proves the compactness part of Theorem~\ref{thm:lowfreq-rigidity}.  All Fourier projections use the sharp cutoff notation introduced in the first section.  In the proof of $L^2$ precompactness, $N$ denotes the low-frequency cutoff and $M$ the high-frequency cutoff.  Low-frequency tail smallness is a Fourier condition, but its role is precisely to rule out hidden infinite mass at zero frequency; see Figure~\ref{fig:lowfreq-schematic}.

\begin{figure}[H]
\centering
\includegraphics[width=0.86\textwidth]{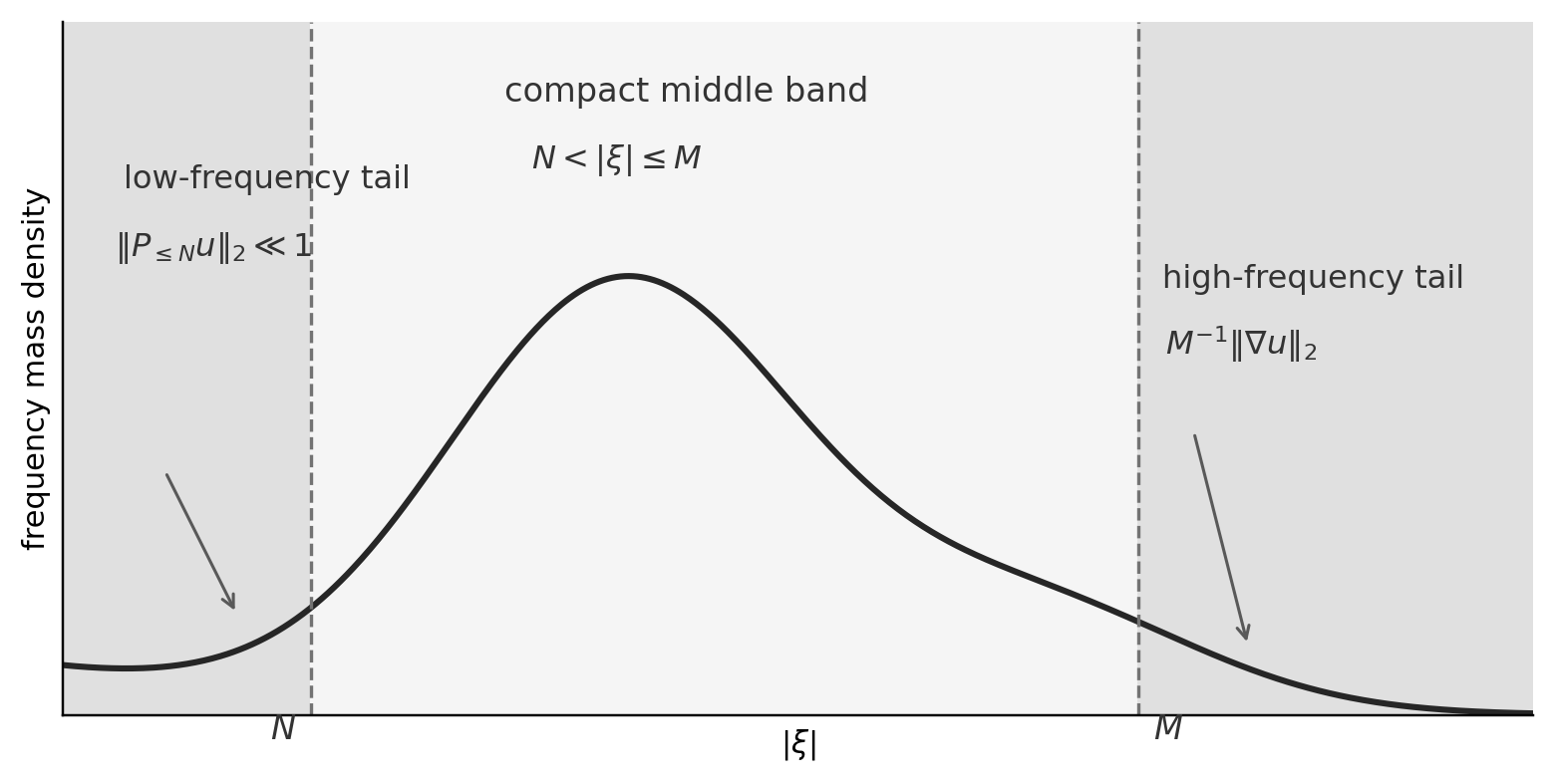}
\caption{Role of low-frequency tail smallness in frequency space.  The low-frequency tail is controlled by assumption, the high-frequency tail by $M^{-1}\|\na u\|_2$, and on the middle band the $L^2$ distance is controlled by $N^{-1}\|\na(f-g)\|_2$.  Thus $\dot H^1$ precompactness upgrades to $L^2$ precompactness.}
\label{fig:lowfreq-schematic}
\end{figure}

First, the low-frequency entrance excludes infinite mass at the origin of frequency space.

\begin{lemma}[Finite mass]\label{lem:lowfreq-finite-mass}
Let $u$ be a nonzero compact critical element satisfying \eqref{eq:lowfreq-small-intro}.  Then
\begin{equation}\label{eq:finite-mass-conclusion}
  \sup_{t\in I}\|u(t)\|_2<\infty.
\end{equation}
In particular $u(t)\in H^1(\R^3)$, and the mass $M(u)$ is finite and conserved.
\end{lemma}

\begin{proof}
Choose $N_0\in(0,1)$ with $\sup_t\|P_{\le N_0}u(t)\|_2<\infty$.  For the high-frequency part,
\[
  \|P_{>N_0}u(t)\|_2^2
  =\int_{|\xi|>N_0}|\widehat u(t,\xi)|^2\dd\xi
  \le N_0^{-2}\int_{|\xi|>N_0}|\xi|^2|\widehat u(t,\xi)|^2\dd\xi
  \le N_0^{-2}\|\na u(t)\|_2^2.
\]
The threshold gradient bound then gives \eqref{eq:finite-mass-conclusion}.
\end{proof}

Finite mass alone is not enough for the later center-of-mass argument.  We also need to upgrade $\dot H^1$ precompactness to $L^2$ precompactness.

\begin{lemma}[$L^2$ precompactness]\label{lem:lowfreq-L2-precompact}
Assume that $K_z$ is precompact in $\dot H^1$ and that \eqref{eq:lowfreq-small-intro} holds.  Then $K_z$ is precompact in $L^2(\R^3)$.
\end{lemma}

\begin{proof}
Set $f_t(x)=u(t,x+z(t)\ez)$.  Translations do not change $|\widehat f|$, hence
\[
  \sup_t\|P_{\le N}f_t\|_2=\sup_t\|P_{\le N}u(t)\|_2.
\]
Given $\eps>0$, first choose $N\in(0,1)$ such that
\begin{equation}\label{eq:lowfreq-eps}
  \sup_t\|P_{\le N}f_t\|_2\le\eps.
\end{equation}
Then choose $M>N$ sufficiently large so that
\begin{equation}\label{eq:highfreq-eps}
  M^{-1}\sup_t\|\na f_t\|_2\le\eps.
\end{equation}
For $f,g\in K_z$, decompose
\[
  f-g=P_{\le N}(f-g)+P_{N<\cdot\le M}(f-g)+P_{>M}(f-g).
\]
By \eqref{eq:lowfreq-eps}, \eqref{eq:highfreq-eps}, and Plancherel,
\begin{align}
  \|f-g\|_2
  &\le 2\eps+N^{-1}\|\na(f-g)\|_2+2M^{-1}\sup_t\|\na f_t\|_2 \notag\\
  &\le 4\eps+N^{-1}\|\na(f-g)\|_2.\label{eq:L2-by-Hdot}
\end{align}
Since $K_z$ is precompact in $\dot H^1$, there exist $g_1,\dots,g_J\in K_z$ such that for every $f\in K_z$ some $j$ satisfies
\[
  \|\na(f-g_j)\|_2\le N\eps.
\]
Then \eqref{eq:L2-by-Hdot} yields $\|f-g_j\|_2\le5\eps$.  Hence $K_z$ is totally bounded in $L^2$, and precompactness follows.
\end{proof}

We shall use the following elementary compactness fact to convert functional-space compactness into a uniform spatial tail estimate.

\begin{lemma}[Uniform tail]\label{lem:compact-tail}
Let $\mathcal K\subset L^2(\R^3)$ be precompact in $L^2$.  Then for every $\eta>0$ there exists $L<\infty$ such that
\begin{equation}\label{eq:L2-compact-tail}
  \sup_{f\in\mathcal K}\int_{|x|\ge L}|f(x)|^2\dd x\le\eta.
\end{equation}
Similarly, if $\mathcal K\subset\dot H^1$ is precompact in $\dot H^1$, then $L$ can be chosen so that the gradient tail is uniformly small.
\end{lemma}

\begin{proof}
We first prove the $L^2$ case.  Given $\eta>0$, set
\[
  \delta=\frac{\eta^{1/2}}{3}.
\]
By precompactness there exist $g_1,\ldots,g_J\in L^2(\R^3)$ such that
\[
  \mathcal K\subset \bigcup_{j=1}^J B_{L^2}(g_j,\delta),
  \qquad B_{L^2}(g_j,\delta)=\{h:\|h-g_j\|_2<\delta\}.
\]
For each $j$, choose $L_j<\infty$ so that
\[
  \int_{|x|\ge L_j}|g_j(x)|^2\dd x\le\delta^2.
\]
Let $L=\max_{1\le j\le J}L_j$.  For $f\in\mathcal K$, choose $j(f)$ with $\|f-g_{j(f)}\|_2<\delta$.  Then
\begin{align*}
  \left(\int_{|x|\ge L}|f(x)|^2\dd x\right)^{1/2}
  &\le \left(\int_{|x|\ge L}|f(x)-g_{j(f)}(x)|^2\dd x\right)^{1/2} \\
  &\quad +\left(\int_{|x|\ge L}|g_{j(f)}(x)|^2\dd x\right)^{1/2} \\
  &\le \|f-g_{j(f)}\|_2+\delta\le2\delta.
\end{align*}
Thus
\[
  \int_{|x|\ge L}|f(x)|^2\dd x\le4\delta^2=\frac49\eta<\eta.
\]
Taking the supremum over $f\in\mathcal K$ gives \eqref{eq:L2-compact-tail}.  If $\mathcal K\subset\dot H^1$ is precompact, then
\[
  \nabla\mathcal K:=\{\na f:f\in\mathcal K\}
\]
is precompact in $L^2(\R^3;\C^3)$.  Applying the same argument to $\nabla\mathcal K$ gives
\[
  \sup_{f\in\mathcal K}\int_{|x|\ge L}|\na f(x)|^2\dd x\le\eta.
\]
\end{proof}

Combining $L^2$ precompactness, $\dot H^1$ precompactness, and the soliton-like tail gives the enhanced compactness needed for the finite-mass virial argument.

\begin{proposition}[Enhanced compactness]\label{prop:enhanced-tightness}
Under the hypotheses of Lemma~\ref{lem:lowfreq-L2-precompact}, for every $\eta>0$ there exists $L_\eta<\infty$ such that
\begin{equation}\label{eq:enhanced-tightness}
  \sup_{t\in I}\int_{|x-z(t)\ez|\ge L_\eta}
  \bigl(|u(t,x)|^2+|\na u(t,x)|^2+|u(t,x)|^6\bigr)\dd x\le\eta.
\end{equation}
\end{proposition}

\begin{proof}
Let $f_t(x)=u(t,x+z(t)\ez)$.  By $L^2$ precompactness and Lemma~\ref{lem:compact-tail}, choose $L_1$ such that
\[
  \sup_t\int_{|x|\ge L_1}|f_t(x)|^2\dd x\le \eta/3.
\]
By $\dot H^1$ precompactness, choose $L_2$ such that
\[
  \sup_t\int_{|x|\ge L_2}|\na f_t(x)|^2\dd x\le \eta/3.
\]
By soliton-like tightness \eqref{eq:tightness}, choose $L_3$ such that
\[
  \sup_t\int_{|x|\ge L_3}|f_t(x)|^6\dd x\le \eta/3.
\]
Taking $L_\eta=\max\{L_1,L_2,L_3\}$ and translating back gives \eqref{eq:enhanced-tightness}.
\end{proof}

Negative regularity is a common sufficient condition for low-frequency tail smallness.

\begin{proposition}[Negative regularity entrance]\label{prop:negative-regularity}
Let $u$ be a nonzero soliton-like cylindrical compact critical element.  If there exists $s\in(0,1]$ such that
\[
  \sup_{t\in I}\||\na|^{-s}u(t)\|_2<\infty,
\]
then $u$ satisfies \eqref{eq:lowfreq-small-intro}.
\end{proposition}

\begin{proof}
For $0<N\le1$,
\[
  \|P_{\le N}u(t)\|_2^2
  =\int_{|\xi|\le N}|\widehat u(t,\xi)|^2\dd\xi
  \le N^{2s}\int_{|\xi|\le N}|\xi|^{-2s}|\widehat u(t,\xi)|^2\dd\xi
  \le N^{2s}\||\na|^{-s}u(t)\|_2^2.
\]
Taking the supremum over $t$ and letting $N\downarrow0$ proves the claim.
\end{proof}

Proposition~\ref{prop:negative-regularity} shows that low-frequency tail smallness covers the usual negative-regularity entrance.  Theorem~\ref{thm:lowfreq-rigidity} may therefore be viewed as a low-frequency version of negative-regularity rigidity.

\subsection{Excluding finite-mass zero-momentum compact elements}

We next prove that finite mass, enhanced compactness, and zero axial momentum automatically yield a sublinear axial drift that can be excluded by localized virial.  This is the dynamical core of the low-frequency rigidity theorem.

The axial Galilean zero-momentum gauge was defined in the introduction.  We only record the formulas needed below.  For $\xi=\xi_3\ez$, the axial Galilean transform $u\mapsto u^\xi$ preserves the equation and satisfies, for $u\in H^1$,
\[
  M(u^\xi)=M(u),
  \qquad P_z(u^\xi)=P_z(u)+\xi_3M(u),
\]
and
\[
  E(u^\xi)=E(u)+\xi_3P_z(u)+\frac12\xi_3^2M(u).
\]
Choosing $\xi_3=-P_z(u)/M(u)$ gives $P_z(u^\xi)=0$.  Moreover,
\[
  \|\na u^\xi(t)\|_2^2
  =\|\na u(t)\|_2^2-\frac{P_z(u)^2}{M(u)}
  \le\|\na u(t)\|_2^2,
\]
and $E(u^\xi)\le E(u)$, so the threshold condition is preserved.  Cylindrical symmetry is also preserved.  If $z(t)$ is an axial center for $u$, then $z_\xi(t)=z(t)+2\xi_3t$ is an axial center for $u^\xi$.  Enhanced compactness is preserved because, with $y=x-2\xi t$,
\[
  |x-z_\xi(t)\ez|=|y-z(t)\ez|,
  \qquad |u^\xi(t,x)|=|u(t,y)|,
\]
and
\[
  |\na u^\xi(t,x)|^2
  \le2|\na u(t,y)|^2+2|\xi|^2|u(t,y)|^2.
\]

Let $\Phi\in C^\infty(\R)$ be odd and satisfy
\[
  \Phi(s)=s\quad(|s|\le1),
  \qquad |\Phi(s)|\le2,
  \qquad 0\le\Phi'(s)\le1,
  \qquad \Phi'(s)=0\quad(|s|\ge2).
\]
Set
\[
  \Phi_R(z)=R\Phi(z/R),
  \qquad B_R(t)=\int_{\R^3}\Phi_R(x_3)|u(t,x)|^2\dd x.
\]
By the local mass conservation law \eqref{eq:mass-local},
\begin{align*}
  B_R'(t)
  &=\int \Phi_R(x_3)\pa_t\rho(t,x)\dd x
  =-2\int \Phi_R(x_3)\operatorname{div}j(t,x)\dd x
  =2\int \Phi_R'(x_3)j_z(t,x)\dd x.
\end{align*}
If $P_z(u)=\int j_z\dd x=0$, then
\[
  B_R'(t)=2\int (\Phi_R'(x_3)-1)j_z(t,x)\dd x.
\]
Since $\Phi_R'(x_3)-1=0$ for $|x_3|\le R$,
\begin{equation}\label{eq:BR-prime-tail}
  |B_R'(t)|\le C
  \left(\int_{|x_3|\ge R}|u(t,x)|^2\dd x\right)^{1/2}
  \left(\int_{|x_3|\ge R}|\na u(t,x)|^2\dd x\right)^{1/2}.
\end{equation}

The localized center-of-mass identity gives the axial drift bound.

\begin{lemma}[Sublinear drift]\label{lem:sublinear-drift}
Let $u$ be a finite-mass solution satisfying enhanced compactness \eqref{eq:enhanced-tightness} and $P_z(u)=0$.  Let
\[
  Z_T^*=\sup_{0\le t\le T}|z(t)|.
\]
Then
\begin{equation}\label{eq:sublinear-drift}
  \frac{Z_T^*}{T}\to0,
  \qquad T\to\infty.
\end{equation}
\end{lemma}

\begin{proof}
Fix $\eta>0$ and let $L_\eta$ be given by \eqref{eq:enhanced-tightness}.  Choose $t_T\in[0,T]$ with $|z(t_T)|\ge Z_T^*-1$ and set
\begin{equation}\label{eq:R-center-mass}
  R=4(Z_T^*+L_\eta+1).
\end{equation}
For $0\le t\le T$, if $|x_3|\ge R$, then
\[
  |x-z(t)\ez|\ge |x_3|-|z(t)|\ge 4(Z_T^*+L_\eta+1)-Z_T^*\ge L_\eta.
\]
Thus \eqref{eq:BR-prime-tail} and enhanced compactness imply
\[
  |B_R'(t)|\le C\eta,
  \qquad 0\le t\le T,
\]
and hence
\begin{equation}\label{eq:BR-difference-small}
  |B_R(t_T)-B_R(0)|\le C\eta T.
\end{equation}
If $|x-z(t)\ez|\le L_\eta$, then $|x_3|\le |z(t)|+L_\eta<R/2$, so $\Phi_R(x_3)=x_3$.  Therefore
\begin{align}
  |B_R(t)-M(u)z(t)|
  &\le \int_{|x-z(t)\ez|\le L_\eta}|x_3-z(t)||u(t,x)|^2\dd x \notag\\
  &\quad +\int_{|x-z(t)\ez|>L_\eta}|\Phi_R(x_3)-z(t)||u(t,x)|^2\dd x \notag\\
  &\le L_\eta M(u)+C(R+Z_T^*)\eta.\label{eq:BR-compare-center}
\end{align}
Using \eqref{eq:BR-compare-center} at $t=t_T$ and $t=0$, and then \eqref{eq:BR-difference-small}, we obtain
\begin{align*}
  M(u)(Z_T^*-1)
  &\le M(u)|z(t_T)| \\
  &\le |B_R(t_T)|+L_\eta M(u)+C(R+Z_T^*)\eta \\
  &\le |B_R(0)|+C\eta T+L_\eta M(u)+C(R+Z_T^*)\eta \\
  &\le M(u)|z(0)|+2L_\eta M(u)+C(R+Z_T^*)\eta+C\eta T.
\end{align*}
Since $R\le C(Z_T^*+L_\eta+1)$, it follows that
\[
  (M(u)-C\eta)Z_T^*\le C_{u,\eta}+C\eta T.
\]
First choose $\eta$ so that $C\eta\le M(u)/2$, divide by $T$, and let $T\to\infty$:
\[
  \limsup_{T\to\infty}\frac{Z_T^*}{T}\le C_u\eta.
\]
Letting $\eta\downarrow0$ proves \eqref{eq:sublinear-drift}.
\end{proof}

Once sublinear drift is known, the finite-mass endpoint estimate excludes the compact element.

\begin{proposition}[Zero-momentum rigidity]\label{prop:finite-mass-zero-momentum}
Let $u$ be a nonzero soliton-like cylindrical compact critical element.  If $M(u)<\infty$, $u$ satisfies enhanced compactness \eqref{eq:enhanced-tightness}, and $P_z(u)=0$, then no such $u$ exists.
\end{proposition}

\begin{proof}
By Lemma~\ref{lem:sublinear-drift}, $Z_T^*=o(T)$.  Choose $\eta_T\downarrow0$ sufficiently slowly so that
\[
  \frac{L_{\eta_T}}{T}\to0.
\]
For example, choose it by a diagonal argument on integer intervals.  Let
\[
  R_T=4(L_{\eta_T}+Z_T^*+1)+T^{1/2}.
\]
Then $R_T/T\to0$ and $R_T\ge4(L_{\eta_T}+|z(t)|+1)$ for $0\le t\le T$.  If $|x|\ge R_T/2$, then
\[
  |x-z(t)\ez|\ge |x|-|z(t)|\ge R_T/2-Z_T^*\ge L_{\eta_T}.
\]
By the virial error estimate \eqref{eq:error-hardied} and enhanced compactness \eqref{eq:enhanced-tightness}, for $T$ large,
\[
  |\Err_{R_T,0}(t)|\le c_1,
  \qquad 0\le t\le T.
\]
Using \eqref{eq:MRZ-derivative} and \eqref{eq:c1-positive},
\[
  M'_{R_T,0}(t)\ge3c_1,
  \qquad 0\le t\le T.
\]
Thus
\[
  3c_1T\le M_{R_T,0}(T)-M_{R_T,0}(0).
\]
But the finite-mass endpoint estimate \eqref{eq:mass-endpoint} gives
\[
  |M_{R_T,0}(T)|+|M_{R_T,0}(0)|\le C_uR_T=o(T),
\]
a contradiction.
\end{proof}

\begin{proof}[Proof of Theorem~\ref{thm:lowfreq-rigidity}]
By Lemma~\ref{lem:lowfreq-finite-mass}, Lemma~\ref{lem:lowfreq-L2-precompact}, and Proposition~\ref{prop:enhanced-tightness}, the solution has finite mass and enhanced compactness.  If $P_z(u)=0$, Proposition~\ref{prop:finite-mass-zero-momentum} gives a contradiction.  In general, apply the axial Galilean transform defined in the introduction with $\xi_3=-P_z(u)/M(u)$.  The formulas above show that the transformed solution has zero axial momentum, keeps the threshold condition, and still satisfies enhanced compactness.  Proposition~\ref{prop:finite-mass-zero-momentum} then gives the contradiction.
\end{proof}

\section{Proof of the conditional scattering theorem}\label{sec:scattering}

We now prove Theorem~\ref{thm:conditional-scattering} and clarify the role of the two entrances in the scattering reduction.

\begin{proof}[Proof of Theorem~\ref{thm:conditional-scattering}]
Assume for contradiction that there is a cylindrically symmetric datum $u_0\in\dot H^1(\R^3)$ satisfying \eqref{eq:threshold} whose solution does not scatter.  By the compactness reduction, Proposition~\ref{prop:compactness-reduction}, there exists a nonzero soliton-like cylindrical compact critical element $u_c$.

By the hypothesis of the theorem, $u_c$ satisfies one of the two entrances.  If it satisfies the best fixed-axis window condition \eqref{eq:Aopt-subdiff-main}, then Theorem~\ref{thm:window-rigidity} excludes $u_c$.  If it satisfies low-frequency tail smallness \eqref{eq:lowfreq-small-intro}, then Theorem~\ref{thm:lowfreq-rigidity} excludes $u_c$.  Both alternatives contradict the existence of $u_c$.  Hence failure of scattering is impossible.  Proposition~\ref{prop:LWP} gives \eqref{eq:global-lifespan-intro}, \eqref{eq:scatter-norm-intro}, and the scattering states in \eqref{eq:scattering-states-intro}.  The Duhamel representation \eqref{eq:duhamel-scatter-intro} follows by taking the limit of the integral formula in $\dot H^1$ as the upper endpoint tends to $\pm\infty$.
\end{proof}

\section*{Acknowledgements}
The author would like to thank the supporting agencies for their continued support.

\section*{Funding}
The author was supported by the Project ``Research on Nonlinear Partial Differential Equations'' (No. 2024KYCXTD018), the Special Projects in Key Areas of Guangdong Province (No. ZDZX1088), and the Fund of Guangzhou Municipal Science and Technology (No. 202102080428).

\section*{Declaration of Competing Interest}
The authors declare that they have no known competing financial interests or personal relationships that could have appeared to influence the work reported in this paper.


\begin{thebibliography}{99}

\bibitem{Strichartz1977}
R. S. Strichartz,
\newblock Restrictions of Fourier transforms to quadratic surfaces and decay of solutions of wave equations,
\newblock \emph{Duke Math. J.} 44 (1977), 705--714.

\bibitem{CazenaveWeissler1990}
T. Cazenave and F. B. Weissler,
\newblock The Cauchy problem for the critical nonlinear Schr\"odinger equation in $H^s$,
\newblock \emph{Nonlinear Anal.} 14 (1990), 807--836.

\bibitem{KeelTao1998}
M. Keel and T. Tao,
\newblock Endpoint Strichartz estimates,
\newblock \emph{Amer. J. Math.} 120 (1998), 955--980.

\bibitem{Gerard1998}
P. G\'erard,
\newblock Description du d\'efaut de compacit\'e de l'injection de Sobolev,
\newblock \emph{ESAIM Control Optim. Calc. Var.} 3 (1998), 213--233.

\bibitem{BahouriGerard1999}
H. Bahouri and P. G\'erard,
\newblock High frequency approximation of solutions to critical nonlinear wave equations,
\newblock \emph{Amer. J. Math.} 121 (1999), 131--175.

\bibitem{Keraani2001}
S. Keraani,
\newblock On the defect of compactness for the Strichartz estimates of the Schr\"odinger equations,
\newblock \emph{J. Differential Equations} 175 (2001), 353--392.

\bibitem{Bourgain1999}
J. Bourgain,
\newblock Global wellposedness of defocusing critical nonlinear Schr\"odinger equation in the radial case,
\newblock \emph{J. Amer. Math. Soc.} 12 (1999), 145--171.

\bibitem{Visan2007}
M. Visan,
\newblock The defocusing energy-critical nonlinear Schr\"odinger equation in higher dimensions,
\newblock \emph{Duke Math. J.} 138 (2007), 281--374.

\bibitem{CollianderKeelStaffilaniTakaokaTao2008}
J. Colliander, M. Keel, G. Staffilani, H. Takaoka and T. Tao,
\newblock Global well-posedness and scattering for the energy-critical nonlinear Schr\"odinger equation in $\R^3$,
\newblock \emph{Ann. of Math. (2)} 167 (2008), 767--865.

\bibitem{KenigMerle2006}
C. E. Kenig and F. Merle,
\newblock Global well-posedness, scattering and blow-up for the energy-critical focusing non-linear Schr\"odinger equation in the radial case,
\newblock \emph{Invent. Math.} 166 (2006), 645--675.

\bibitem{DuyckaertsMerle2009}
T. Duyckaerts and F. Merle,
\newblock Dynamics of threshold solutions for energy-critical NLS,
\newblock \emph{Geom. Funct. Anal.} 18 (2009), 1787--1840.

\bibitem{KillipVisan2010}
R. Killip and M. Visan,
\newblock The focusing energy-critical nonlinear Schr\"odinger equation in dimensions five and higher,
\newblock \emph{Amer. J. Math.} 132 (2010), 361--424.

\bibitem{Dodson2019}
B. Dodson,
\newblock Global well-posedness and scattering for the focusing, cubic Schr\"odinger equation in dimension $d=4$,
\newblock \emph{Ann. Sci. \'Ec. Norm. Sup\'er.} 52 (2019), 139--180.

\bibitem{HolmerRoudenko2008}
J. Holmer and S. Roudenko,
\newblock A sharp condition for scattering of the radial 3D cubic nonlinear Schr\"odinger equation,
\newblock \emph{Comm. Math. Phys.} 282 (2008), 435--467.

\bibitem{DuyckaertsHolmerRoudenko2008}
T. Duyckaerts, J. Holmer and S. Roudenko,
\newblock Scattering for the non-radial 3D cubic nonlinear Schr\"odinger equation,
\newblock \emph{Math. Res. Lett.} 15 (2008), 1233--1250.

\bibitem{TaoVisanZhang2008}
T. Tao, M. Visan and X. Zhang,
\newblock Minimal-mass blowup solutions of the mass-critical NLS,
\newblock \emph{Forum Math.} 20 (2008), 881--919.

\bibitem{Aubin1976}
T. Aubin,
\newblock Probl\`emes isop\'erim\'etriques et espaces de Sobolev,
\newblock \emph{J. Differential Geometry} 11 (1976), 573--598.

\bibitem{Talenti1976}
G. Talenti,
\newblock Best constant in Sobolev inequality,
\newblock \emph{Ann. Mat. Pura Appl.} 110 (1976), 353--372.

\bibitem{Stein1970}
E. M. Stein,
\newblock \emph{Singular Integrals and Differentiability Properties of Functions},
\newblock Princeton University Press, 1970.

\bibitem{LiebLoss2001}
E. H. Lieb and M. Loss,
\newblock \emph{Analysis}, 2nd ed.,
\newblock Graduate Studies in Mathematics, vol. 14, American Mathematical Society, 2001.

\bibitem{Cazenave2003}
T. Cazenave,
\newblock \emph{Semilinear Schr\"odinger Equations},
\newblock Courant Lecture Notes in Mathematics, vol. 10, American Mathematical Society, 2003.

\bibitem{Tao2006}
T. Tao,
\newblock \emph{Nonlinear Dispersive Equations: Local and Global Analysis},
\newblock CBMS Regional Conference Series in Mathematics, vol. 106, American Mathematical Society, 2006.

\bibitem{Grafakos2014}
L. Grafakos,
\newblock \emph{Classical Fourier Analysis}, 3rd ed.,
\newblock Graduate Texts in Mathematics, vol. 249, Springer, 2014.

\bibitem{LinaresPonce2015}
F. Linares and G. Ponce,
\newblock \emph{Introduction to Nonlinear Dispersive Equations}, 2nd ed.,
\newblock Universitext, Springer, 2015.

\end{thebibliography}
\end{document}